\documentclass[11pt]{article}
\usepackage[a4paper, total={6in, 8in}]{geometry}
\usepackage{amsmath,amssymb, amsthm}
\usepackage{array}
\usepackage{booktabs}
\usepackage{mathtools}
\usepackage{amsthm}
\usepackage{xfrac}
\usepackage{faktor}
\usepackage{stackrel}
\usepackage{multicol}
\usepackage{centernot}
\usepackage{bbm}
\usepackage{tikz}
\usepackage{xcolor}
\usepackage{rotating}
\usepackage{longtable}
\usepackage{tikz-cd}
\usepackage{graphicx}
\usepackage{ulem}
\usepackage{indentfirst}
\usepackage{lipsum}
\usetikzlibrary{arrows}

\newgeometry{tmargin=3cm, bmargin=3cm, lmargin=3cm, rmargin=3cm}

\newcommand\blfootnote[1]{%
  \begingroup
  \renewcommand\thefootnote{}\footnote{#1}%
  \addtocounter{footnote}{-1}%
  \endgroup
}

\def\XXint#1#2#3{{\setbox0=\hbox{$#1{#2#3}{\int}$ }
\vcenter{\hbox{$#2#3$ }}\kern-.6\wd0}}

\theoremstyle{definition}

\newtheorem{thm}{Theorem}
\newtheorem{prop}{Proposition}[section]
\newtheorem{lem}[prop]{Lemma}

\newtheorem{con}{Conjecture}

\newtheorem{rem}[prop]{Remark}

\newtheorem{cor}[prop]{Corollary}
\newtheorem*{xrem}{Remark}

\newcommand{\e}{\varepsilon}

\newcommand{\E}{\mathbb{E}}

\newcommand{\R}{\mathbb{R}}

\newcommand{\N}{\mathbb{N}}
\newcommand{\Prb}{\mathbb{P}}

\newcommand{\1}{\mathbf{1}}

\newcommand{\sign}{\text{sign}}

\title{Stability of Khintchine inequalities with optimal constants between the second and the $p$-th moment for $p \ge 3$}
\author{Jacek Jakimiuk}
\date{}

\begin{document}
    \maketitle{}

    \begin{abstract}
        We give a strengthening of the classical Khintchine inequality between the second and the $p$-th moment for $p \ge 3$ with optimal constant by adding a deficit depending on the vector of coefficients of the Rademacher sum. 
    \end{abstract}

    \section{Introduction}

    The classical Khintchine inequality is the inequality of the form $\|S\|_p \le C_{p, q}\|S\|_q$, where $S$ is a Rademacher sum, i.e. $S = \sum_{i=1}^na_i\e_i$ for i.i.d. Rademacher random variables $\e_1, \ldots, \e_n$, $\Prb(\e_i = 1) = \Prb(\e_i = -1) = \frac 1 2$, and any real numbers $a_1, \ldots, a_n$. Here $p > q > 0$ and $C_{p, q} > 0$ is a universal constant depending only on $p$ and $q$, while $\|X\|_r = \left(\E|X|^r\right)^{\frac 1 r}$ stands for the $r$-th moment of a random variable $X$. Research on these inequalities originates form Khintchine \cite{Kh} and Littlewood \cite{L}.

    The natural problem of finding the optimal value of constant $C_{p, q}$ has been studied for several decades. Khintchine found the optimal value of $C_{p, 2}$ for even integers $p$ without stating the optimality, see \cite{Kh}. Optimal $C_{p, 2}$ for $p \ge 3$ were found by Whittle \cite{W} and independently by Young \cite{Y}, with later alternative proofs and extensions in \cite{H,P,FHJSZ}. Optimal $C_{2, 1}$ was originally found by Szarek \cite{S} and the proof was simplified in \cite{H,LO,O}. Haagerup in \cite{H} found the optimal constants $C_{p, 2}$ for all $p > 2$ and $C_{2, p}$ for all $0 < p < 2$. His proof was simplified for $0 < p < 2$ by Nazarov and Podkorytov \cite{NP} and for $2 < p < 3$ by Mordhorst \cite{M}. Optimal $C_{p, q}$ for even integers $p > q$ was found by Czerwi\'nski \cite{C} in the case when $q$ divides $p$ and by Nayar and Oleszkiewicz \cite{NO} in the general case. In all known cases it turns out that $C_{p, q} = \min\left\{\frac{\|G\|_p}{\|G\|_q}, \frac{\|S_2\|_p}{\|S_2\|_q}\right\}$, where $G$ is a standard Gaussian variable and $S_2 = \frac{\e_1 + \e_2}{\sqrt{2}}$. In the cases where $C_{p, q}$ is the Gaussian constant the optimality is asymptotic, i.e. there is no vector of coefficients for which the equality is attained, but there is a sequence of vectors achieving the equality in the limit. Then one may ask for the optimal constant $C_{p, q, n}$ such that $\|S\|_p \le C_{p, q, n}\|S\|_q$ for any Rademacher sum $S = \sum_{i=1}^na_i\e_i$ with fixed $n$. It is known due to Komorowski \cite{K} (up to our best knowledge) that $C_{p, 2, n} = \frac{\|S_n\|_p}{\|S_n\|_2}$ for $p \ge 3$, where $S_n = \frac{1}{\sqrt{n}}\sum_{i=1}^n\e_i$ corresponds to the diagonal vector of coefficients.

    Another natural question is about stability versions of the optimal Khintchine inequalities, that is inequalities strengthened by some positive deficit depending onf some notion of distance from the extremizer. This question was recently studied by De, Diakonikolas, Servedio \cite{DDS} (the authors called it ``robust of Khitchine inequality"), Eskenazis, Nayar, Tkocz \cite{ENT2} and Melbourne, Roberto \cite{MR} in the case of $C_{2, 1}$. They derived inequalities of the form
    \begin{align*}
        \E|S| \ge \frac{1}{\sqrt{2}}\E|\e_1 + \e_2| - c\left|a - a^{(2)}\right|,
    \end{align*}
    where $S = \sum_{i=1}^na_i\e_i$, $\sum_{i=1}^na_i^2 = 1$, $a = (a_1, \ldots, a_n)$, $a^{(2)} = \left(\frac{1}{\sqrt{2}}, \frac{1}{\sqrt{2}}, 0, \ldots, 0\right)$ and $c$ is a universal constant. Eskenazis, Nayar and Tkocz in \cite{ENT2} used this stability result in the context of extremal projections of convex bodies. They also investigated in \cite{ENT1} a different kind of stability, i.e. inequality of the form
    \begin{align*}
        \E\left|\sum_{i=1}^na_iX_i\right| \ge \E\left|\frac{X_1 + X_2}{\sqrt{2}}\right|
    \end{align*}
    for $X_1, \ldots, X_n$ i.i.d. symmetric close to Rademacher distribution and $\sum_{i=1}^na_i^2 = 1$, $a_i^2 \le \frac 1 2$ for all $i$, which they called ``distributional stability".

    In the present paper we shall investigate the stability of the optimal constants $C_{p, 2}$ and $C_{p, 2, n}$ for $p \ge 3$. Our main idea is to obtain deficit by improving some of the known proofs of Khintchine inequalities with these constants. This case is closely related with the Schur order, thus now we provide a short reminder of this theory.

    \subsection{Schur order}

    For a vector $x = (x_1, \ldots, x_n)$ with non-negative coefficients summing up to 1 let $x_1^*, \ldots, x_n^*$ be a non-increasing rearrangement of the coordinates $x_1, \ldots, x_n$. Given two vectors $x$, $y$ as above we say that $x$ is lesser than $y$ in the Schur order or that $x$ is majorized by $y$ iff $\forall_{1 \le k \le n}\sum_{i=1}^kx_i^* \le \sum_{i=1}^ky_i^*$ to be denoted $x \prec y$. In general the assumptions of non-negativity of coordinates and their sum equal 1 are not necessary, but it will always be the case in this paper. We define a $T$-transformation as a linear function of the form
    \begin{align*}
        T_{j, k}(x) = (x_1, \ldots, x_{j-1}, (1-\lambda)x_j + \lambda x_k, x_{j+1}, \ldots, x_{k-1}, \lambda x_j + (1-\lambda)x_k, x_{k+1}, \ldots, x_n),
    \end{align*}
    where $\lambda \in [0, 1]$. In other words, it is any linear function which takes two coordinates and brings them closer to each other preserving their sum. We also say that a function $F : \Delta_{n-1} \to \R$, where $\Delta_{n-1} = \left\{(x_1, \ldots, x_n) : \sum_{i=1}^nx_i = 1, \forall_ix_i \ge 0\right\}$ is a simplex, is Schur convex iff $x \prec y$ implies $F(x) \le F(y)$. 

    The connection between Schur order and Khintchine inequality is given by the following proposition, proved by Eaton \cite{E} with the main corollary due to Komorowski \cite{K}.

    \begin{prop}\label{schurmon}
        Let $\e_1, \ldots, \e_n$ be i.i.d. Rademacher random variables, $a_1, \ldots, a_n, b_1, \ldots, b_n$ be real numbers satisfying $\sum_{i=1}^na_i^2 = \sum_{i=1}^nb_i^2 = 1$ and $(a_1^2, \ldots, a_n^2) \prec (b_1^2, \ldots, b_n^2)$, $\Phi : \R \to \R$ be even and such that $\Phi''$ is convex. Then
        \begin{align*}
            \E\Phi\left(\sum_{i=1}^nb_i\e_i\right) \le \E\Phi\left(\sum_{i=1}^na_i\e_i\right).
        \end{align*}
        In particular for $p \ge 3$ we get
        \begin{align}\label{pschurmon}
            \E\left|\sum_{i=1}^nb_i\e_i\right|^p \le \E\left|\sum_{i=1}^na_i\e_i\right|^p,
        \end{align}
        equivalently the function $x \mapsto -\E\left|\sum_{i=1}^n\sqrt{x_i}\e_i\right|^p$ is Schur convex.
    \end{prop}

    Since easily $\left(\frac 1 n, \ldots, \frac 1 n\right) \prec x \prec (1, 0, \ldots, 0)$ for any $x \in \Delta_{n-1}$, the above proposition proves that under $\E|S|^2 = 1$ the $p$-th moment $\E|S|^p$ of a Rademacher sum $S$ is maximized for $S = \frac{1}{\sqrt{n}}\sum_{i=1}^n\e_i$ for any $p \ge 3$.

    \subsection{Main results}

    In this paper we shall investigate two kinds of stability. Let $p \ge 3$, $S = \sum_{i=1}^na_i\e_i$ be a Rademacher sum and $G$ be a standard Gaussian variable. For simplicity we shall assume that $\E|S|^2 = \sum_{i=1}^na_i^2 = 1$. We know that the constant $C_{p, 2} = \|G\|_p$ is asymptotically achieved by vectors $a = (a_1, \ldots, a_n)$ satisfying anti-concentration assumptions of CLT, particularly by $a_i = \frac{1}{\sqrt{n}}$ for all $i$. We prove a strengthening of the inequality $\E|S|^p \le \E|G|^p$ by some deficit depending on a vector $a$ which is always positive. We call it ``Gaussian stability".

    \begin{thm}\label{gauss}
        Let $p \ge 3$, $\e_1, \ldots, \e_n$ be i.i.d. Rademacher random variables, $G$ be a standard Gaussian variable, $a_1 \ge \ldots \ge a_n \ge 0$ be such that $\sum_{i=1}^na_i^2 = 1$, $S = \sum_{i=1}^na_i\e_i$ be a Rademacher sum. Then there exists a constant $C_p > 0$, depending only on $p$, such that
        \begin{align*}
            \E|S|^p \le \E|G|^p - C_p\sum_{i=1}^na_i^4.
        \end{align*}
    \end{thm}

    \begin{xrem}
        It is clear that the deficit $C_p\sum_{i=1}^na_i^4$ is increasing with respect to the Schur order, due to Schur convexity of $x \mapsto x_1^2 + \ldots + x_n^2$. Also the result of Theorem \ref{gauss} is optimal in the following sense: $q = 4$ is the minimal value of $q$ for which there exists a constant $C(p, q) > 0$ such that the inequality
        \begin{align*}
            \E|S|^p \le \E|G|^p - C(p, q)\sum_{i=1}^na_i^q
        \end{align*}
        holds for every $n, a_1, \ldots, a_n$. We shall prove it in Section 5.
    \end{xrem}

    The second kind of stability, which we call ``diagonal stability", is the stability of the constant $C_{p, 2, n} = \|S_n\|_p$ for fixed $n$, where $S_n = \frac{1}{\sqrt{n}}\sum_{i=1}^n\e_i$. Here we separate the cases $p > 3$ and $p = 3$.

    \begin{thm}\label{diag}
        Let $p > 3$, $\e_1, \ldots, \e_n$ be i.i.d. Rademacher variables, numbers $a_1 \ge \ldots \ge a_n \ge 0$ be such that $\sum_{i=1}^na_i^2 = 1$, $S = \sum_{i=1}^na_i\e_i$, $S_n = \frac{1}{\sqrt{n}}\sum_{i=1}^n\e_i$. Then there exists a constant $C_p > 0$, depending only on $p$, such that
        \begin{align*}
            \E|S|^p \le \E|S_n|^p - C_p\sum_{i=1}^n\left(a_i^2 - \frac 1 n\right)^2.
        \end{align*}
    \end{thm}

    \begin{xrem}
        Similarly to the previous remark, the deficit $C_p\sum_{i=1}^n\left(a_i^2 - \frac 1 n\right)^2$ is increasing with respect to the Schur order.
    \end{xrem}

    In the case $p = 3$ we get somehow worse result. The reason why it is worse lies deeply in the proof, but might be related with the criticality of $p = 3$, as for $p < 3$ it is unclear whether $S_n$ maximizes $\E|S|^p$. In the following results (and throughout the whole paper) we use the notation $x_+ = \max\{x, 0\}$ for $x \in \R$.

    \begin{thm}\label{crit}
        Let $\e_1, \ldots, \e_n$ be i.i.d. Rademacher variables, $a_1 \ge \ldots \ge a_n \ge 0$ be such that $\sum_{i=1}^na_i^2 = 1$, $S = \sum_{i=1}^na_i\e_i$, $S_n = \frac{1}{\sqrt{n}}\sum_{i=1}^n\e_i$. Then there exists a constant $C_3 > 0$ such that
        \begin{align*}
            \E|S|^3 \le \E|S_n|^3 - \frac{C_3}{a_1^2 + \frac 1 n}\sum_{i=1}^n\sqrt{\frac 1 n \left(a_i^2 + a_n^2 - \frac 1 n\right)}\left(a_i^2 - \frac 1 n\right)_+\left(\frac 1 n - a_{n+1-i}^2\right)_+.
        \end{align*}
    \end{thm}

    The above deficit is clearly 0 only if $a_i = \frac{1}{\sqrt{n}}$ for all $i$. However, it is not comparable with the quantity $\sum_{i=1}^n\left(a_i^2 - \frac 1 n\right)^2$ and seems to be far from being optimal.\\

    In the Gaussian case the main idea of the proof is to write $G = \sum_{i=1}^na_iG_i$ for independent Gaussian variables $G_1, \ldots, G_n$ and bound the difference
    \begin{align*}
        \E\left|\sum_{i=1}^na_iG_i\right|^p - \E\left|\sum_{i=1}^na_i\e_i\right|^p
    \end{align*}
    from below by exchanging Rademachers for Gaussians one by one, keeping track of the deficit. In the diagonal case the main idea of the proof is quite similar, we proceed by nearing pairs of coefficients until all of them become equal and again keep track of the deficit at every step. In both cases we rely on improving the convexity arguments of the papers \cite{W,P,FHJSZ,K}. 
    
    In Section 2 we investigate the improved convexity together with providing concentration bounds we shall need later. In Section 3 we prove Theorem \ref{gauss}, whereas in Section 4 we prove Theorems \ref{diag} and \ref{crit}. In Section 5 we provide some remarks and open questions. 
    
    Throughout the whole paper $G$ is a standard Gaussian random variable, $\e_1, \ldots, \e_n$ are i.i.d. Rademacher variables, $S_n = \frac{1}{\sqrt{n}}\sum_{i=1}^n\e_i$, $a^{(n)} = \frac{1}{\sqrt{n}}(1, \ldots, 1)$ and $b^{(n)} = \frac 1 n (1, \ldots, 1)$.

    \section{Preliminary results}

    \subsection{The function $\psi_s$}

    Let $p \ge 3$. For $s \in \R$ we define a function $\psi_s : [0, \infty) \to \R$ by the formula 
    \begin{align*}
        \psi_s(t) = 2\E\left|s + \e\sqrt{t}\right|^p = \left|s + \sqrt{t}\right|^p + \left|s - \sqrt{t}\right|^p,
    \end{align*}
    where $\e$ is a Rademacher random variable. In the works of Whittle \cite{W}, Pinelis \cite{P} and Figiel, Hitczenko, Johnson, Schechtman, Zinn \cite{FHJSZ} the crucial property of $\psi_s$ is the convexity of $\psi_s''$. In order to prove Khintchine inequality with some positive deficit we shall need something stronger, i.e. some bounds from below on $\psi_s''$ or convexity of some perturbations of $\psi_s$. We begin with derivation of an integral formula for $\psi_s''$ for $p > 3$ together with nice explicit formula for $p = 3$.

    \begin{lem}\label{formula}
        Let $t > 0$. For $p > 3$ we have
        \begin{align}\label{formula1}
            \psi_s''(t) = \frac{p(p-1)(p-2)(p-3)}{4t^{\frac 3 2}}\int_0^{\sqrt{t}}\int_x^{\sqrt{t}}\int_{|s|-y}^{|s|+y}|z|^{p-4}dzdydx
        \end{align}
        and for $p = 3$ we have
        \begin{align}\label{formula2}
            \psi_s''(t) = \frac{3(t - s^2)_+}{2t^{\frac 3 2}}.
        \end{align}
    \end{lem}

    \begin{proof}
        Obviously $\psi_s = \psi_{-s}$, hence we may assume $s \ge 0$. By simple computation we have
        \begin{align*}
            \psi_s'(t) = \frac{p}{2\sqrt{t}}\left(\left|s + \sqrt{t}\right|^{p-1}\sign\left(s + \sqrt{t}\right) - \left|s - \sqrt{t}\right|^{p-1}\sign\left(s - \sqrt{t}\right)\right)
        \end{align*}
        and
        \begin{align*}
            \psi_s''(t) &= \frac{p(p-1)\left(\left|s + \sqrt{t}\right|^{p-2} + \left|s - \sqrt{t}\right|^{p-2}\right)}{4t} \\ &- \frac{p\left(\left|s + \sqrt{t}\right|^{p-2}\left(s + \sqrt{t}\right) - \left|s - \sqrt{t}\right|^{p-2}\left(s - \sqrt{t}\right)\right)}{4t^{\frac 3 2}}.
        \end{align*}
        Denote $\varphi_s(t) = \frac{4t^{\frac 3 2}}{p(p-1)}\psi_s''(t)$. By using $\frac{d}{dx}|x|^{\alpha} = \alpha|x|^{\alpha - 1}\sign(x)$ for $\alpha > 0$ and the fundamental theorem of calculus several times we get
        \begin{align*}
            \varphi_s(t) &= \sqrt{t}\left(\left|s + \sqrt{t}\right|^{p-2} + \left|s - \sqrt{t}\right|^{p-2}\right) - \int_{s - \sqrt{t}}^{s + \sqrt{t}}|x|^{p-2}dx \\ &= \int_0^{\sqrt{t}}\left(\left|s + \sqrt{t}\right|^{p-2} - |s + x|^{p-2} - |s - x|^{p-2} + \left|s - \sqrt{t}\right|^{p-2}\right)dx \\ &= (p-2)\int_0^{\sqrt{t}}\left(\int_{s+x}^{s + \sqrt{t}}|y|^{p-3}\sign(y)dy - \int_{s - \sqrt{t}}^{s-x}|y|^{p-3}\sign(y)dy\right)dx \\ &= (p-2)\int_0^{\sqrt{t}}\int_x^{\sqrt{t}}\left(|y - s|^{p-3}\sign(y - s) - |-y - s|^{p-3}\sign(-y - s)\right)dydx \\ &= (p-2)(p-3)\int_0^{\sqrt{t}}\int_x^{\sqrt{t}}\int_{-s-y}^{-s+y}|z|^{p-4}dzdydx.
        \end{align*}
        The last equality is valid for $p > 3$ and fails for $p = 3$ due to the integrability of $z \mapsto |z|^{p-4}$ around 0. All preceding steps are valid also for $p = 3$. Obviously changing $-s$ to $s$ does not change the last integral and thus we proved the formula \eqref{formula1}.

        For $p = 3$ we already have
        \begin{align*}
            \varphi_s(t) = \sqrt{t}\left(\left|s + \sqrt{t}\right| + \left|s - \sqrt{t}\right|\right) - \int_{s - \sqrt{t}}^{s + \sqrt{t}}|x|dx.
        \end{align*}
        For $s < \sqrt{t}$ it gives
        \begin{align*}
            \varphi_s(t) = 2t - \int_0^{\sqrt{t}+s}xdx - \int_0^{\sqrt{t}-s}xdx = t - s^2
        \end{align*}
        and for $s \ge \sqrt{t}$ it gives
        \begin{align*}
            \varphi_s(t) = 2s\sqrt{t} - \int_{s-\sqrt{t}}^{s+\sqrt{t}}xdx = 0,
        \end{align*}
        which proves the formula \eqref{formula2}.
    \end{proof}

    Now we shall use the formula \eqref{formula1} to derive some bounds from below on $\psi_s''$. Let us begin with the case $p \ge 4$. Then the function $z \mapsto |z|^{p-4}$ is non-decreasing on $[0, \infty)$ and hence clearly the formula \eqref{formula1} is non-decreasing in $|s|$. Thus the following corollary holds.

    \begin{cor}\label{ezbound}
        For $p \ge 4$ and $t > 0$ we have $\psi_s''(t) \ge \psi_0''(t) = \frac{p(p-2)}{2}t^{\frac{p-4}{2}}$. Equivalently the function $\psi_s(t) - \psi_0(t) = \psi_s(t) - 2t^{\frac p 2}$ is convex in $t$.
    \end{cor}

    It turns out that this simple bound is not optimal, for example using it we would only prove Theorem \ref{gauss} with $\sum_{i=1}^na_i^p$ instead of $\sum_{i=1}^na_i^4$. The next lemma provides another lower bound for $\psi_s''(t)$, which is independent of $t$ and hence much stronger for small values of $t$.

    \begin{lem}\label{p4bound}
        For $p > 4$ we have
        \begin{align*}
            \psi_s''(t) \ge \frac{3p(p-1)(p-2)(p-3)}{64}|s|^{p-4},
        \end{align*}
        equivalently the function $\psi_s(t) - \frac{3p(p-1)(p-2)(p-3)}{128}|s|^{p-4}t^2$ is convex in $t$.
    \end{lem}

    \begin{proof}
        Assume for simplicity $s \ge 0$ and denote $\varphi_s(t) = \int_0^{\sqrt{t}}\int_x^{\sqrt{t}}\int_{s-y}^{s+y}|z|^{p-4}dzdydx$. By Fubini's Theorem we get
        \begin{align*}
            \varphi_s(t) &= \int_0^{\infty}\int_0^{\infty}\int_{-\infty}^{\infty}|z|^{p-4}\1_{|z-s| \le y}\1_{x \le y \le \sqrt{t}}dzdydx \\ &= \int_{-\infty}^{\infty}|z|^{p-4}\left|\left\{(x, y) \in \R^2 : 0 \le x \le y, |z - s| \le y \le \sqrt{t}\right\}\right|dz \\ &= \int_{s - \sqrt{t}}^{s + \sqrt{t}}|z|^{p-4}\int_{|z-s|}^{\sqrt{t}}\int_0^y1dxdydz = \int_{s - \sqrt{t}}^{s + \sqrt{t}}|z|^{p-4}\int_{|z-s|}^{\sqrt{t}}ydydz \\ &= \frac 1 2 \int_{s - \sqrt{t}}^{s + \sqrt{t}}|z|^{p-4}\left(t - |z-s|^2\right)dz \ge \frac 1 2\int_s^{s + \frac{\sqrt{t}}{2}}|z|^{p-4}\left(t - |z-s|^2\right)dz \\ &\ge \frac{\sqrt{t}}{4} \cdot |s|^{p-4} \cdot \frac{3t}{4} = \frac{3t^{\frac 3 2}|s|^{p-4}}{16},
        \end{align*}
        where in the last inequality we used $z \ge s$ and $|z-s|^2 \le \frac t 4$. Hence
        \begin{align*}
            \psi_s''(t) &= \frac{p(p-1)(p-2)(p-3)}{4t^{\frac 3 2}}\varphi_s(t) \ge \frac{p(p-1)(p-2)(p-3)}{4t^{\frac 3 2}} \cdot \frac{3t^{\frac 3 2}|s|^{p-4}}{16} \\ &= \frac{3p(p-1)(p-2)(p-3)}{64}|s|^{p-4}.
        \end{align*}
    \end{proof}

    Both bounds obtained above for $p \ge 4$ use the fact that the function $z \mapsto |z|^{p-4}$ is non-decreasing on $[0, \infty)$. For $3 < p < 4$ this function is non-increasing, which results with the different bound.

    \begin{lem}\label{p3bound}
        For $3 < p < 4$ we have
        \begin{align*}
            \psi_s''(t) \ge \frac{p(p-1)(p-2)(p-3)}{6}\left(|s| + \sqrt{t}\right)^{p-4}.
        \end{align*}
    \end{lem}

    \begin{proof}
        Again we assume $s \ge 0$ and denote $\varphi_s(t) = \int_0^{\sqrt{t}}\int_x^{\sqrt{t}}\int_{s-y}^{s+y}|z|^{p-4}dzdydx$. For $z$ under the integral we have $|z| \le s + \sqrt{t}$ and hence $|z|^{p-4} \ge \left(s + \sqrt{t}\right)^{p-4}$. Thus
        \begin{align*}
            \varphi_s(t) &\ge \int_0^{\sqrt{t}}\int_x^{\sqrt{t}}2y\left(s + \sqrt{t}\right)^{p-4}dydx = \left(s + \sqrt{t}\right)^{p-4}\int_0^{\sqrt{t}}\left(t - x^2\right)dx = \left(s + \sqrt{t}\right)^{p-4} \cdot \frac 2 3 t^{\frac 3 2}.
        \end{align*}
        Together with formula \eqref{formula1} it ends the proof.
    \end{proof}

    \subsection{Concentration bounds on Rademacher sums}

    In the upcoming sections we shall usually take $\psi_s$ with random parameter $s$ equal to some Rademacher sum. Bound from Lemma \ref{p3bound} and formula \eqref{formula2} for $p = 3$ clearly are useful only for small values of $s$. Thus we need some concentration bounds on Rademacher sums.

    We begin with recalling a simple proposition, being an immediate consequence of Chernoff bound.

    \begin{prop}\label{chernoff}
        Let $a_1, \ldots, a_n \ge 0$ be such that $\sum_{i=1}^na_i^2 \le 1$ and $S = \sum_{i=1}^na_i\e_i$. Then
        \begin{align*}
            \Prb(|S| \le 2) \ge 1 - 2e^{-1}.
        \end{align*}
    \end{prop}

    \begin{proof}
        By a well known Chernoff bound for Rademacher sums we have
        \begin{align*}
            \Prb(S > t) \le e^{-\frac{t^2}{4}}
        \end{align*}
        for any $t > 0$. It suffices to apply it with $t = 2$ and use the symmetry of $S$.
    \end{proof}

    Proposition \ref{chernoff} gives us bound on the probability that a Rademacher sum is contained in some constant size interval. We shall also need concentration on the intervals of size depending on the largest coefficient of the Rademacher sum. We shall extensively use the following proposition.

    \begin{prop}[Lemma 1 in \cite{DG}]\label{BE}
        Let $1 \ge a \ge a_1 \ge a_2 \ge \ldots \ge a_n \ge 0$ be such that $\sum_{i=1}^na_i^2 \le 1$, $S = \sum_{i=1}^na_i\e_i$. Then there exists a universal constant $c \ge \frac{3}{16}$ such that
        \begin{align*}
            \Prb(|S| \le a) \ge ca.
        \end{align*}
    \end{prop}

    The above result was proved with unspecified constant in \cite{KKMO}, but we cite it from Dzindzalieta, G\"otze \cite{DG}. Their proof relies on splitting into cases ``$a_1$ is large" and ``$a_1$ is small", using known concentration results of Rademacher sums in the first case and carefully analyzed Berry-Esseen bound in the second case. 

    \begin{rem}\label{BEopt}
        Note that in Proposition \ref{BE} the assumption $a \ge a_1$ is crucial. Consider $n$ odd, $a_1 = \ldots = a_n = \frac{1}{\sqrt{n}}$. Then $\Prb\left(|S| < \frac{1}{\sqrt{n}}\right) = 0$. It shows that for $a < a_1$ we cannot hope for any reasonable lower bound on $\Prb(|S| \le a)$.
    \end{rem}

    We shall also need concentration bounds for Rademacher sums with additional Gaussian component. Let $S' = S + bG$, where $S$ is the same as in Proposition \ref{chernoff} or \ref{BE}, $G$ is independent of $S$ and $b \ge 0$ is such that $b^2 + \sum_{i=1}^ka_i^2 \le 1$. Denote $S'_n = \frac{b}{\sqrt{n}}\sum_{i=1}^n\rho_i + S$, where $\rho_1, \ldots, \rho_n$ are i.i.d. Rademacher variables independent of $S$, then $S'_n$ by CLT converges in distribution to $S'$ and the statements of Proposition \ref{chernoff} and \ref{BE} are satisfied for sums $S'_n$ (in the case of Proposition \ref{BE} we must require that $n$ is sufficiently large). Thus by passing to the limit with $n \to \infty$ we obtain the following corollary.

    \begin{cor}\label{Gaussmod}
        Propositions \ref{chernoff} and \ref{BE} hold true with $S$ replaced by $S + bG$, where $G$ is independent of $S$ and $b \ge 0$ is such that $b^2 + \sum_{i=1}^ka_i^2 \le 1$.
    \end{cor}

    \section{Gaussian stability}

    Recall that the strategy of the proof relies on exchanging Rademacher variables for Gaussians one by one, estimating the deficit on each of these replacement and summing them up to get the estimate of the total deficit. The following lemma provides a lower bound for single exchange.

    \begin{lem}\label{lemgauss}
        Let $\e, \e_1, \ldots, \e_k$ be i.i.d. Rademacher variables, $G$, $G'$ be i.i.d. standard Gaussian variables independent of $(\e_1, \ldots, \e_k)$, $b \ge 0$ and $\frac{1}{\sqrt{2}} \ge a \ge a_1 \ge \ldots \ge a_k \ge 0$ be such that $a^2 + b^2 + \sum_{i=1}^ka_i^2 = 1$, $S = bG' + \sum_{i=1}^ka_i\e_i$. Then
        \begin{align*}
            \E|S + aG|^p - \E|S + a\e|^p \ge C_pa^4,
        \end{align*}
        where $C_p$ is the same constant as in Theorem \ref{gauss}.
    \end{lem}

    \begin{proof}
        We have
        \begin{align*}
            \E|S + aG|^p - \E|S + a\e|^p &= \E|S + a\e|G||^p - \E|S + a\e|^p \\ &= \frac 1 2\left(\E\psi_S\left(a^2G^2\right) - \E\psi_S\left(a^2\right)\right) = \frac 1 2\E\left[\psi_S\left(a^2G^2\right) - \psi_S\left(a^2\right)\right].
        \end{align*}
        Thus it suffices to prove that $\E\left[\psi_S\left(a^2G^2\right) - \psi_S\left(a^2\right)\right] \ge 2C_pa^4$. Now we split the proof into four cases.\\

        \textit{Case $p > 4$:} By Lemma \ref{p4bound} we know that the function $f_S(t) = \psi_S(t) - c_p|S|^{p-4}t^2$, where $c_p = \frac{3p(p-1)(p-2)(p-3)}{128}$, is convex in $t$ on $(0, \infty)$. By Jensen's inequality we have
        \begin{align*}
            \E_Gf_S(a^2G^2) \ge f_S\left(\E_Ga^2G^2\right) = f_S(a^2),
        \end{align*}
        which by conditioning on $S$ and simple algebra yields
        \begin{align*}
            \E\left[\psi_S(a^2G^2) - \psi_S(a^2)\right] \ge c_p\E\left[|S|^{p-4}(a^4G^4 - a^4)\right] = c_p\E|S|^{p-4}a^4\E[G^4 - 1] = 2c_pa^4\E|S|^{p-4}.
        \end{align*}
        It remains to prove that $\E|S|^{p-4}$ can be bounded from below by a constant depending only on $p$. For $p \ge 6$ we simply get
        \begin{align*}
            \E|S|^{p-4} \ge \left(\E|S|^2\right)^{\frac{p-4}{2}} = (1 - a^2)^{\frac{p-4}{2}} \ge 2^{\frac{4-p}{2}},
        \end{align*}
        which ends the proof with $C_p = 2^{\frac{4-p}{2}}c_p$. Now assume $4 < p < 6$. Let $\rho_1, \ldots, \rho_n$ be i.i.d. Rademacher variables independent of $(\e_1, \ldots, \e_k)$ and $S_n = b\sum_{i=1}^n\frac{\rho_i}{\sqrt{n}} + \sum_{i=1}^ka_i\e_i$. Then by Khintchine inequalities with known optimal constants (c.f. \cite{H}) there exists a constant $\widetilde{C}_p > 0$ such that $\E|S_n|^{p-4} \ge \widetilde{C}_p(1 - a^2)^{\frac{p-4}{2}} \ge \widetilde{C}_p2^{\frac{4-p}{2}}$. As by CLT $S_n$ converges to $S$ in distribution and the second moments of $S_n$ and $S$ are uniformly bounded, we get $\E|S_n|^{p-4} \xrightarrow{n \to \infty} \E|S|^{p-4}$ (recall that $p - 4 < 2$ in the considered case) and
        \begin{align*}
            \E|S|^{p-4} \ge \widetilde{C}_p2^{\frac{4-p}{2}}.
        \end{align*}
        This yields the statement with $C_p = 2^{\frac{4-p}{2}}c_p\widetilde{C}_p$.\\

        \textit{Baby case $p = 4$:} To avoid the term $0^0$ which would appear in the above reasoning if $p = 4$ and $\Prb(S = 0) > 0$, we proceed as above with $f_S(t) = \psi_S(t) - 2t^2$ and using Corollary \ref{ezbound} instead of Lemma \ref{p4bound}. It gives
        \begin{align*}
            \E\left[\psi_S(a^2G^2) - \psi_S(a^2)\right] \ge 2a^4\E[G^4 - 1] = 4a^4.
        \end{align*}

        \textit{Case $3 < p < 4$:} By the Taylor expansion with integral remainder in the first equality and substitution $u = a^2t$ in the second, and using the convention $\int_1^x = -\int_x^1$ for $x < 1$, we have
        \begin{align}\label{taylor1}
            \E\left[\psi_S(a^2G^2) - \psi_S(a^2)\right] &= \E\left[\psi_S'(a^2)\left(a^2G^2 - a^2\right) + \int_{a^2}^{a^2G^2}\psi_S''(u)\left(a^2G^2 - u\right)du\right] \\ &= a^4\E\int_1^{G^2}\psi_S''(a^2t)(G^2 - t)dt \ge a^4\int_1^{\infty}\E\left[\1_{t \le G^2}(G^2 - t)\psi_S''(a^2t)\right]dt. \notag
        \end{align}
        By a known Gaussian tail estimate $\int_t^{\infty}e^{-\frac{u^2}{2}}du \ge \left(\frac 1 t - \frac{1}{t^3}\right)e^{-\frac{t^2}{2}}$, for $t \ge 1$ we get
        \begin{align}\label{g1}
            \E\left[1_{t \le G^2}(G^2 - t)\right] \ge \Prb\left(G^2 \ge t + 1\right) \ge \frac{2}{\sqrt{2\pi}}e^{-\frac{t+1}{2}}\left(\frac{1}{\sqrt{t+1}} - \frac{1}{(t + 1)^{\frac 3 2}}\right).
        \end{align}
        Using Lemma \ref{p3bound}, Proposition \ref{chernoff} with Corollary \ref{Gaussmod}, $p-4 < 0$, $a \le 1$ and substituting $c_p = \frac{p(p-1)(p-2)(p-3)}{6} > 0$ we obtain
        \begin{align}\label{s1}
            \E\psi_S''(a^2t) &\ge c_p\E\1_{|S| \le 2}\left||S| + a\sqrt{t}\right|^{p-4} \ge c_p\Prb(|S| \le 2)\left(2 + a\sqrt{t}\right)^{p-4} \\ &\ge c_p\left(1 - 2e^{-1}\right)\left(2 + \sqrt{t}\right)^{p-4}. \notag
        \end{align}
        By \eqref{g1} and \eqref{s1} we get
        \begin{align*}
            \int_1^{\infty}\E\left[\1_{t \le G^2}(G^2 - t)\psi_S''(a^2t)\right]dt \ge \frac{2c_p\left(1 - 2e^{-1}\right)}{\sqrt{2\pi}}\int_1^{\infty}\frac{t}{(t+1)\sqrt{t+1}}\left(2 + \sqrt{t}\right)^{p-4}e^{-\frac{t+1}{2}}dt,
        \end{align*}
        which is positive and depends only on $p$. Thus by \eqref{taylor1} we have
        \begin{align*}
            \E\left[\psi_S(a^2G^2) - \psi_S(a^2)\right] \ge 2C_pa^4.
        \end{align*}

        \textit{Case $p = 3$:} Analogously to the previous case we get
        \begin{align}\label{taylor2}
            \E\left[\psi_S(a^2G^2) - \psi_S(a^2)\right] \ge a^4\int_1^{\infty}\E\left[\1_{t \le G^2}(G^2 - t)\psi_S''(a^2t)\right]dt.
        \end{align}
        Now by formula \eqref{formula2} and Proposition \ref{BE} with Corollary \ref{Gaussmod} we have
        \begin{align*}
            \E\psi_S''(a^2t) = \E\left[\frac{3(a^2t - S^2)_+}{2a^3t^{\frac 3 2}}\right] \ge \Prb(|S| \le a)\frac{3(t - 1)}{2at^{\frac 3 2}} \ge \frac{9(t - 1)}{32t^{\frac 3 2}},
        \end{align*}
        which together with \eqref{g1} and \eqref{taylor2} gives
        \begin{align*}
            \E\left[\psi_S(a^2G^2) - \psi_S(a^2)\right] \ge \frac{9a^4}{16\sqrt{2\pi}}\int_1^{\infty}\frac{t}{(t+1)\sqrt{t+1}}\frac{t-1}{t\sqrt{t}}e^{-\frac{t+1}{2}}dt =: 2C_pa^4.
        \end{align*}
    \end{proof}

    We are now ready to prove Theorem \ref{gauss}.

    \begin{proof}[Proof of Theorem \ref{gauss}]
        Suppose that $a_1 \le \frac{1}{\sqrt{2}}$. Let $G_1, \ldots, G_n$ be i.i.d. standard Gaussian variables independent of $(\e_1, \ldots, \e_n)$. Denoting $S^i = \sum_{j=1}^{i-1}a_jG_j + \sum_{j=i+1}^na_j\e_j$ for $i = 1, \ldots, n$ we have
        \begin{align*}
            \E|G|^p - \E|S|^p = \sum_{i=1}^n\left(\E|S^i + a_iG_i|^p - \E|S^i + a_i\e_i|^p\right).
        \end{align*}
        Thus it suffices to prove that $\E|S^i + a_iG_i|^p - \E|S^i + a_i\e_i|^p \ge C_pa_i^4$ for some constant $C_p > 0$. Now Lemma \ref{lemgauss} easily ends the proof.

        Now consider the remaining case $a_1 > \frac{1}{\sqrt{2}}$. For $n = 1$ there is nothing to prove. Assume $n \ge 2$, then obviously $a_i < \frac{1}{\sqrt{2}}$ for $i > 1$ and we may perform a sequence of $T$-transformations on the vector $\left(a_1^2, \ldots, a_n^2\right)$ nearing $a_1$ to the other coefficients until $a_1 = \frac{1}{\sqrt{2}}$. Denote the resulting vector of coefficients by $(b_1, \ldots, b_n)$, then $\frac{1}{\sqrt{2}} = b_1 \ge b_2 \ge \ldots \ge b_n$ and from the previous part of the proof we have
        \begin{align*}
            \E\left|\sum_{i=1}^nb_i\e_i\right|^p \le \E|G|^p - C_p\sum_{i=1}^nb_i^4.
        \end{align*}
        By Proposition \ref{schurmon} we have $\E\left|\sum_{i=1}^nb_i\e_i\right|^p \ge \E|S|^p$, moreover $\frac 1 4 \le \sum_{i=1}^nb_i^4 \le \sum_{i=1}^na_i^4 \le 1$ implies $\sum_{i=1}^na_i^4 \le 4\sum_{i=1}^nb_i^4$. Hence
        \begin{align*}
            \E|S|^p \le \E|G|^p - \frac{C_p}{4}\sum_{i=1}^na_i^4.
        \end{align*}
    \end{proof}

    \section{Diagonal stability}

    Let us formalize the idea of nearing pairs of coefficients mentioned in the introduction. Given a vector $\left(a_1^2, \ldots, a_n^2\right)$, we perform a sequence of $T$-transformations transforming it into the vector $b^{(n)}$, we bound from below the deficit on each of these $T$-transformations and we sum up these bounds to get the statement. The following lemma provides a bound for a single $T$-transformation.

    \begin{lem}\label{lemdiag}
        Let $n \ge 3$, $a_1 \ge \ldots \ge a_n \ge 0$ be such that $\sum_{i=1}^na_i^2 = 1$, $S = \sum_{i=2}^{n-1}a_i\e_i$, $a = \sqrt{a_1^2 + a_n^2}$, $\mu_0 = \frac{a_n^2}{a^2}$, $\mu = \frac{1}{na^2}$. Then
        \begin{align*}
            \E\left|\sqrt{1-\mu}a\e_1 + \sqrt{\mu}a\e_n + S\right|^p - \E\left|\sqrt{1-\mu_0}a\e_1 + \sqrt{\mu_0}a\e_n + S\right|^p \ge 2C_p\left(a_1^2 - \frac 1 n\right)\left(\frac 1 n - a_n^2\right),
        \end{align*}
        where $C_p$ is the same constant as in Theorem \ref{diag}. For $p > 4$ we additionally assume that $a_1^2 \le 0.9 - \frac 1 n$.
    \end{lem}

    \begin{proof}
        Let $\e$ be a Rademacher variable independent of $(\e_1, \ldots, \e_n)$. For $s, x \in \R$ and $t \in [0, 1]$ define $\psi_{s, x}(t) = \E\psi_s\left(x^2(1 + t\e)\right) = \frac 1 2\left(\psi_s\left(x^2(1 + t)\right) + \psi_s\left(x^2(1 - t)\right)\right)$. We have
        \begin{align*}
            \E\left|\sqrt{1-\mu}a\e_1 + \sqrt{\mu}a\e_n + S\right|^p &= \E\left|\e\left|\sqrt{1-\mu}a\e_1 + \sqrt{\mu}a\e_n\right| + S\right|^p \\ &= \E\left|\e\sqrt{a^2 + 2a^2\sqrt{\mu(1 - \mu)}\e_1\e_n} + S\right|^p \\ &= \frac 1 2 \E\psi_S\left(a^2\left(1 + 2\sqrt{\mu(1 - \mu)}\e_1\e_n\right)\right) \\ &= \frac 1 2 \E\psi_{S, a}\left(2\sqrt{\mu(1 - \mu)}\right),
        \end{align*}
        the same for $\mu_0$ in place of $\mu$. Thus it suffices to prove that
        \begin{align}\label{form1}
            \E\left[\psi_{S, a}\left(2\sqrt{\mu(1 - \mu)}\right) - \psi_{S, a}\left(2\sqrt{\mu_0(1 - \mu_0)}\right)\right] \ge 2C_p\left(a_1^2 - \frac 1 n\right)\left(\frac 1 n - a_n^2\right).
        \end{align}
        Note that since $a_1^2 \ge \frac 1 n \ge a_n^2$, we have $\mu_0 \le \mu \le 1 - \mu_0$ and $\mu_0 \le 1 - \mu \le 1 - \mu_0$, hence $\mu(1 - \mu) \ge \mu_0(1 - \mu_0)$ with equality only for $a_1 = a_n = \frac{1}{\sqrt{n}}$, in which case the statement is trivial. We have
        \begin{align*}
            \frac{d}{dt}\psi_{s, x}(t) &= \frac{x^2}{2}\left(\psi_s'\left(x^2(1 + t)\right) - \psi_s'\left(x^2(1 - t)\right)\right) = \frac{x^2}{2}\int_{x^2(1 - t)}^{x^2(1 + t)}\psi_s''(v)dv \\ &= \frac{x^4}{2}\int_{1 - t}^{1 + t}\psi_s''(x^2u)du,
        \end{align*}
        hence \eqref{form1} is equivalent to
        \begin{align*}
            a^4\int_{2\sqrt{\mu_0(1 - \mu_0)}}^{2\sqrt{\mu(1 - \mu)}}\int_{1 - t}^{1 + t}\E\psi_S''(a^2u)dudt \ge 4C_p\left(a_1^2 - \frac 1 n\right)\left(\frac 1 n - a_n^2\right).
        \end{align*}
        Now we split the proof into three cases.\\

        \textit{Case $p > 4$:} Denote $c_p = \frac{3p(p-1)(p-2)(p-3)}{64}$. By Lemma \ref{p4bound} we get
        \begin{align*}
            a^4\int_{2\sqrt{\mu_0(1 - \mu_0)}}^{2\sqrt{\mu(1 - \mu)}}\int_{1 - t}^{1 + t}\E\psi_S''(a^2u)dudt &\ge c_pa^4\int_{2\sqrt{\mu_0(1 - \mu_0)}}^{2\sqrt{\mu(1 - \mu)}}\int_{1 - t}^{1 + t}\E|S|^{p-4}dudt \\ &= c_p\E|S|^{p-4}a^4\int_{2\sqrt{\mu_0(1 - \mu_0)}}^{2\sqrt{\mu(1 - \mu)}}2tdt \\ &= 4c_p\E|S|^{p-4}a^4(\mu(1 - \mu) - \mu_0(1 - \mu_0)) \\ &= 4c_p\E|S|^{p-4}a^2(\mu - \mu_0)a^2(1 - \mu - \mu_0) \\ &= 4c_p\E|S|^{p-4}\left(\frac 1 n - a_n^2\right)\left(a_1^2 - \frac 1 n\right).
        \end{align*}
        Similarly to the proof of Lemma \ref{lemgauss}, the quantity $\E|S|^{p-4}$ is bounded from below by a constant depending only on $p$, if only $a^2 \le c$ for some $c < 1$, which is guaranteed by $a_1^2 \le 0.9 - \frac 1 n$ and $a_n^2 \le \frac 1 n$. It yields the statement.\\

        \textit{Baby case $p = 4$:} Again we consider this case separately to avoid the term $0^0$. Analogously to the previous case, using Corollary \ref{ezbound} instead of Lemma \ref{p4bound}, we get
        \begin{align*}
            a^4\int_{2\sqrt{\mu_0(1 - \mu_0)}}^{2\sqrt{\mu(1 - \mu)}}\int_{1 - t}^{1 + t}\E\psi_S''(a^2u)dudt &\ge 16\left(\frac 1 n - a_n^2\right)\left(a_1^2 - \frac 1 n\right).
        \end{align*}

        \textit{Case $3 < p < 4$:} Denote $c_p = \frac{p(p-1)(p-2)(p-3)}{6}$. By Lemma \ref{p3bound}, Proposition \ref{chernoff}, $a \le 1$ and $2\mu(1 - \mu) \le \frac 1 2$ we get
        \begin{align*}
            a^4\int_{2\sqrt{\mu_0(1 - \mu_0)}}^{2\sqrt{\mu(1 - \mu)}}\int_{1 - t}^{1 + t}\E\psi_S''(a^2u)dudt &\ge c_pa^4\int_{2\sqrt{\mu_0(1 - \mu_0)}}^{2\sqrt{\mu(1 - \mu)}}\int_{1 - t}^{1 + t}\E\left(|S| + a\sqrt{u}\right)^{p-4}dudt \\ &\ge c_pa^4\int_{2\sqrt{\mu_0(1 - \mu_0)}}^{2\sqrt{\mu(1 - \mu)}}\int_{1 - t}^{1 + t}\Prb(|S| \le 2)\left(2 + a\sqrt{u}\right)^{p-4}dudt \\ &\ge c_p\left(1 - 2e^{-1}\right)\left(2 + \sqrt{2}\right)^{p-4}a^4\int_{2\sqrt{\mu_0(1 - \mu_0)}}^{2\sqrt{\mu(1 - \mu)}}\int_{1 - t}^{1 + t}1dudt \\ &= 4c_p\left(1 - 2e^{-1}\right)\left(2 + \sqrt{2}\right)^{p-4}\left(\frac 1 n - a_n^2\right)\left(a_1^2 - \frac 1 n\right).
        \end{align*}
    \end{proof}

    We are now ready to prove Theorem \ref{diag}.

    \begin{proof}[Proof of Theorem \ref{diag}]
        Suppose that $n \ge 3$ and $a_1^2 \le 0.9 - \frac 1 n$. Denote $b_i = a_i^2$ for $i = 1, \ldots, n$. If $b_1 > \frac 1 n > b_n$, we perform the $T$-transformation $(b_1, \ldots, b_n) \mapsto \left(b_1 + b_n - \frac 1 n, b_2, \ldots, b_{n-1}, \frac 1 n\right)$, we apply the permutation of the coordinates of the new vector to obtain the vector $(b_1', \ldots, b_n')$ with $b_1' \ge \ldots \ge b_n'$ and we define $a_i' = \sqrt{b_i'}$. Then we repeat the whole procedure with $(a_1', \ldots, a_n')$ in place of $(a_1, \ldots, a_n)$ until $(a_1, \ldots, a_n) = a^{(n)}$. It is clear that the above procedure terminates in a finite number of steps since each step strictly increases the number of coefficients equal to $\frac{1}{\sqrt{n}}$. As $a_1$ does not increase and $a_n < \frac{1}{\sqrt{n}}$, Lemma \ref{lemdiag} gives us
        \begin{align}\label{singlet}
            \E\left|\sum_{i=1}^na_i'\e_i\right|^p - \E\left|\sum_{i=1}^na_i\e_i\right|^p \ge 2C_p\left(a_1^2 - \frac 1 n\right)\left(\frac 1 n - a_n^2\right).
        \end{align}
        Note that in the notation of Lemma \ref{lemdiag} we have $\sqrt{\mu_0}a = a_n$ and $\sqrt{1 - \mu_0}a = a_1$.

        Now observe that for any real numbers $x$, $y$, $z$ we have
        \begin{align}\label{id1}
            2(x - y)(y - z) = 2\left(xy + yz - xz - y^2\right) = x^2 + z^2 - y^2 - (x - y + z)^2.
        \end{align}
        Applying this identity with $x = a_1^2$, $y = \frac 1 n$, $z = a_n^2$ to \eqref{singlet} gives us
        \begin{align*}
            \E\left|\sum_{i=1}^na_i'\e_i\right|^p - \E\left|\sum_{i=1}^na_i\e_i\right|^p &\ge C_p\left(a_1^4 + a_n^4 - \frac{1}{n^2} - \left(a_1^2 + a_n^2 - \frac 1 n\right)^2\right) \\ &= C_p\left(\sum_{i=1}^na_i^4 - \sum_{i=1}^na_i'^4\right),
        \end{align*}
        where we use the fact that $(a_1', \ldots, a_n')$ is a permutation of the numbers $\frac{1}{\sqrt{n}}, \sqrt{a_1^2 + a_n^2 - \frac 1 n}, a_2,$ $\ldots, a_{n-1}$. Summing it for all the performed $T$-transformations together with the identity
        \begin{align}\label{id2}
            \sum_{i=1}^na_i^4 - \frac 1 n = \sum_{i=1}^na_i^4 - \frac 2 n \sum_{i=1}^na_i^2 + \sum_{i=1}^n\frac{1}{n^2} =  \sum_{i=1}^n\left(a_i^2 - \frac 1 n\right)^2
        \end{align}
        yields the statement.

        Suppose that $n \ge 3$ and $a_1^2 > 0.9 - \frac 1 n$. Then we deal with the large coefficient in the same way as in the proof of Theorem \ref{gauss}. We perform the $T$-transformations nearing $a_1$ to the smaller coefficients, use Schur monotonicity and compare $\left(0.9 - \frac 2 n\right)^2$ with $\sum_{i=1}^n\left(a_i^2 - \frac 1 n\right)^2$ to get the statement at the cost of dividing $C_p$ by at least $\left(0.9 - \frac 2 3\right)^{-2} < 25$.

        Now consider the remaining case $n = 2$. Then, after substituting $x = a_1^2 - \frac 1 2$, the statement takes the form
        \begin{align}\label{n2}
            \left|\sqrt{\frac 1 2 + x} + \sqrt{\frac 1 2 - x}\right|^p + \left|\sqrt{\frac 1 2 + x} - \sqrt{\frac 1 2 - x}\right|^p \le 2^{\frac p 2} - C_px^2
        \end{align}
        for $0 \le x \le \frac 1 2$. Using the Taylor expansion of $y \mapsto \sqrt{1 + y}$ we get
        \begin{align*}
            \left|\sqrt{\frac 1 2 + x} + \sqrt{\frac 1 2 - x}\right|^p + \left|\sqrt{\frac 1 2 + x} - \sqrt{\frac 1 2 - x}\right|^p &= \left|\sqrt{2} - \frac{x^2}{\sqrt{2}} + o(x^2)\right|^p + \left|x\sqrt{2} + o(x^2)\right|^p \\ &= 2^{\frac p 2} - 2^{\frac p 2 - 1}px^2 + o(x^2),
        \end{align*}
        which implies that 
        \begin{align*}
            \left|\sqrt{\frac 1 2 + x} + \sqrt{\frac 1 2 - x}\right|^p + \left|\sqrt{\frac 1 2 + x} - \sqrt{\frac 1 2 - x}\right|^p \le 2^{\frac p 2} - C'_px^2
        \end{align*}
        for $x \in [0, \delta]$ for some $\delta > 0$. If $\delta < x \le \frac 1 2$, then due to the monotonicity of LHS we have 
        \begin{align*}
            \left|\sqrt{\frac 1 2 + x} + \sqrt{\frac 1 2 - x}\right|^p + \left|\sqrt{\frac 1 2 + x} - \sqrt{\frac 1 2 - x}\right|^p \le 2^{\frac p 2} - C'_p\delta^2 \le 2^{\frac p 2} - 4C'_p\delta^2x^2,
        \end{align*}
        which proves \eqref{n2}.
    \end{proof}

    Heading towards the proof of Theorem \ref{crit} we follow the same strategy as for Theorem \ref{diag}. However, it results in a worse stability estimate due to the fact that for $p = 3$ the function $\psi_s''(t)$ is 0 for $t > s^2$. We shall comment on it in the proof of the following lemma, analogous to Lemma \ref{lemdiag}.

    \begin{lem}\label{lemcrit}
        Let $a_1 \ge \ldots \ge a_n \ge 0$ be such that $\sum_{i=1}^na_i^2 = 1$, $S = \sum_{i=2}^{n-1}a_i\e_i$, $a = \sqrt{a_1^2 + a_n^2}$, $\mu_0 = \frac{a_n^2}{a^2}$, $\mu = \frac{1}{na^2}$. Then
        \begin{align*}
            &\E\left|\sqrt{1-\mu}a\e_1 + \sqrt{\mu}a\e_n + S\right|^3 - \E\left|\sqrt{1-\mu_0}a\e_1 + \sqrt{\mu_0}a\e_n + S\right|^3 \\ &\ge C_3\frac{\sqrt{\frac 1 n \left(a^2 - \frac 1 n\right)}}{a^2}\left(a_1^2 - \frac 1 n\right)\left(\frac 1 n - a_n^2\right),
        \end{align*}
        where $C_3$ is the same constant as in Theorem \ref{crit}.
    \end{lem}

    \begin{proof}
        Following the proof of Lemma \ref{lemdiag} we arrive with reduction of the statement to
        \begin{align}\label{form3}
            a^4\int_{2\sqrt{\mu_0(1 - \mu_0)}}^{2\sqrt{\mu(1 - \mu)}}\int_{1 - t}^{1 + t}\E\psi_S''(a^2u)dudt \ge 4C_3\frac{\sqrt{\frac 1 n \left(a^2 - \frac 1 n\right)}}{a^2}\left(a_1^2 - \frac 1 n\right)\left(\frac 1 n - a_n^2\right).
        \end{align}
        Now, by formula \eqref{formula2} we get
        \begin{align*}
            a^4\int_{2\sqrt{\mu_0(1 - \mu_0)}}^{2\sqrt{\mu(1 - \mu)}}\int_{1 - t}^{1 + t}\E\psi_S''(a^2u)dudt &= \frac 3 2 a^4\int_{2\sqrt{\mu_0(1 - \mu_0)}}^{2\sqrt{\mu(1 - \mu)}}\int_{1 - t}^{1 + t}\frac{\E\left(a^2u - S^2\right)_+}{a^3u^{\frac 3 2}}dudt.
        \end{align*}
        By Proposition \ref{BE} we have
        \begin{align}\label{conc}
            \E\left(a^2u - S^2\right)_+ \ge \Prb(|S| \le a)a^2(u - 1)_+ \ge \frac{3}{16}a^3(u - 1)_+,
        \end{align}
        which together with $u^{\frac 3 2} \le \left(\frac 3 2\right)^{\frac 3 2} \le 2$ gives
        \begin{align}\label{mid}
            \frac 3 2 a^4\int_{2\sqrt{\mu_0(1 - \mu_0)}}^{2\sqrt{\mu(1 - \mu)}}\int_{1 - t}^{1 + t}\frac{\E\left(a^2u - S^2\right)_+}{a^3u^{\frac 3 2}}dudt &\ge \frac{9}{64}a^4\int_{2\sqrt{\mu_0(1 - \mu_0)}}^{2\sqrt{\mu(1 - \mu)}}\int_1^{1 + t}(u - 1)dudt \\ &= \frac{9}{64}a^4\int_{2\sqrt{\mu_0(1 - \mu_0)}}^{2\sqrt{\mu(1 - \mu)}}\frac{t^2}{2}dt \notag \\ &= \frac{3}{16}a^4\left((\mu(1 - \mu))^{\frac 3 2} - (\mu_0(1 - \mu_0))^{\frac 3 2}\right). \notag
        \end{align}
        As we see, in \eqref{conc} we get a lower bound not by universal constant, but by an expression linearly small for $u$ close to 1, and by Remark \ref{BEopt} it cannot be significantly improved by this method. After double integration it results with the term of order $(\mu(1 - \mu))^{\frac 3 2} - (\mu_0(1 - \mu_0))^{\frac 3 2}$ instead of $\mu(1 - \mu) - \mu_0(1 - \mu_0)$, which perfectly simplifies in the proof of Lemma \ref{diag}. Note that for any $0 \le b_0 < b_1$ and non-negative function $f$ concave on $[b_0, b_1]$ we have 
        \begin{align}\label{concave}
            \int_{b_0}^{b_1}f(x)dx &= (b_1 - b_0)\int_0^1f((1-\lambda)b_0 + \lambda b_1)d\lambda \ge (b_1 - b_0)\int_0^1((1-\lambda)f(b_0) + \lambda f(b_1))d\lambda \notag \\ &= \frac{(b_1 - b_0)(f(b_1) + f(b_0))}{2} \ge \frac{(b_1 - b_0)f(b_1)}{2}.
        \end{align}
        Continuing \eqref{mid} and using concavity of $x \mapsto \sqrt{x}$ together with \eqref{concave} we have
        \begin{align*}
            \frac{3}{16}a^4\left((\mu(1 - \mu))^{\frac 3 2} - (\mu_0(1 - \mu_0))^{\frac 3 2}\right) &= \frac{a^4}{8}\int_{\mu_0(1 - \mu_0)}^{\mu(1 - \mu)}\sqrt{x}dx \\ &\ge \frac{a^4}{16}(\mu(1 - \mu) - \mu_0(1 - \mu_0))\sqrt{\mu(1 - \mu)} \\ &= \frac{\sqrt{\frac 1 n \left(a^2 - \frac 1 n\right)}}{16a^2}\left(a_1^2 - \frac 1 n\right)\left(\frac 1 n - a_n^2\right).
        \end{align*}
        It proves \eqref{form3} with $C_3 = 2^{-6}$.
    \end{proof}

    Now we are ready to prove Theorem \ref{crit}.

    \begin{proof}[Proof of Theorem \ref{crit}]
        We begin analogously to the proof of Theorem \ref{diag}, using Lemma \ref{lemcrit} instead of Lemma \ref{lemdiag} and thus having
        \begin{align}\label{singlet2}
            \E\left|\sum_{i=1}^na_i'\e_i\right|^3 - \E\left|\sum_{i=1}^na_i\e_i\right|^3 \ge C_3\frac{\sqrt{\frac 1 n\left(a^2 - \frac 1 n\right)}}{a^2}\left(a_1^2 - \frac 1 n\right)\left(\frac 1 n - a_n^2\right).
        \end{align}
        instead of \eqref{singlet}. 
        
        In general the sum of \eqref{singlet2} over all the performed $T$-transformations, on the contrary to the proof of Theorem \ref{diag}, seems to be difficult to represent with a nice formula. We shall bound this sum from below by something more compact and still being a distance between $(a_1, \ldots, a_n)$ and $a^{(n)}$ in some sense. Look at the first $k$ of the performed $T$-transformations, where $k = \max\left\{i : a_i^2 > \frac 1 n > a_{n+1-i}^2\right\}$. As until the $j$-th $T$-transformation at most $j-1$ largest and $j-1$ smallest coefficient could be affected, in the $j$-th step the largest coefficient is no smaller than $a_j$ and the smallest is no larger than $a_{n-j+1}$. Hence at the $j$-th step we add at least $C_3\frac{\sqrt{\frac 1 n\left(a^2 - \frac 1 n\right)}}{a^2}\left(a_j^2 - \frac 1 n\right)\left(\frac 1 n - a_{n-j+1}^2\right)$ to the deficit.
        
        What remains to do is to bound the term $\frac{\sqrt{\frac 1 n\left(a^2 - \frac 1 n\right)}}{a^2}$ in \eqref{singlet2} for the first $k$ of the performed $T$-transformations. As the largest coefficient is bounded from above by $a_1$ and the smallest is bounded from above by $\frac 1 n$, we have $a^2 \le a_1^2 + \frac 1 n$ in the denominator. As the smallest coefficient is bounded from below by $a_n$ and the largest coefficient in the $j$-th $T$-transformation is bounded from below by $a_j$, we have $a^2 - \frac 1 n \ge a_j^2 + a_n^2 - \frac 1 n$ in the numerator of $j$-th $T$-transformation. Thus we can bound the deficit at the $j$-th step by $C_3\frac{\sqrt{\frac 1 n \left(a_j^2 + a_n^2 - \frac 1 n\right)}}{a_1^2 + \frac 1 n}\left(a_j^2 - \frac 1 n\right)_+\left(\frac 1 n - a_{n+1-j}^2\right)_+$. Summing it from $j = 1$ to $j = k$ ends the proof.
    \end{proof}

    \section{Final remarks and open questions}

    \subsection{Optimality of Theorem \ref{gauss}}

    According to the remark after Theorem \ref{gauss}, it is impossible to decrease the exponent 4 in the statement of this theorem. It is an immediate consequence of the following proposition and taking $n \to \infty$ with $a_i = \frac{1}{\sqrt{n}}$ for all $i$.

    \begin{prop}\label{opt}
        Let $p \ge 3$. There exists a universal constant $C(p)$ such that
        \begin{align*}
            \E|S_n|^p \ge \E|G|^p - \frac{C(p)}{n}.
        \end{align*}
    \end{prop}

    \begin{proof}
        We begin with the following claim.
        
        \noindent\textit{Claim.} For any $n \in \N$ we have $\E|S_{2n}|^p \le e^{\frac{p^2}{4n}}\E|S_n|^p$.

        \begin{proof}[Proof of the claim]
            Let $\rho_i = \e_{2i-1} + \e_{2i}$ for $i = 1, 2, \ldots, n$. Then the variables $\rho_1, \ldots, \rho_n$ are i.i.d. and $\rho_i$ has the same distribution as $2X_i\e_i$, where $X_i$ is a symmetric Bernoulli variable, i.e. $\Prb(X_i = 0) = \Prb(X_i = 1) = \frac 1 2$, and $X_1, \ldots, X_n$ are i.i.d. and independent of $(\e_1, \ldots, \e_n)$. By using this representation, denoting $X = \sum_{i=1}^nX_i$ and conditioning on $X_i$ we get
            \begin{align}\label{daniel}
                \E|S_{2n}|^p &= \E\left|\frac{1}{\sqrt{2n}}\sum_{i=1}^n2X_i\e_i\right|^p = \left(\frac 2 n\right)^{\frac p 2}\E_X\E_{\e}\left|\sum_{i=1}^nX_i\e_i\right|^p = \left(\frac 2 n\right)^{\frac p 2}\sum_{k=0}^n\Prb(X = k)\left|\sum_{i=1}^k\e_i\right|^p \notag \\ &= \left(\frac 2 n\right)^{\frac p 2}\sum_{k=0}^n\Prb(X = k)k^{\frac p 2}\E|S_k|^p \le \left(\frac 2 n\right)^{\frac p 2}\E|S_n|^p\sum_{k=0}^n\Prb(X = k)k^{\frac p 2},
            \end{align}
            where the inequality follows from Proposition \ref{schurmon}. Now it follows from Corollary 1 in \cite{A} that
            \begin{align*}
                \sum_{k=0}^n\Prb(X = k)k^{\frac p 2} = \E|X|^{\frac p 2} \le \left(\frac n 2\right)^{\frac p 2}e^{\frac{p^2}{4n}}.
            \end{align*}
            Putting it into \eqref{daniel} yields the claim.
        \end{proof}
        By the claim for any non-negative integer $k$ we have
        \begin{align*}
            \E\left|S_{2^{k+1}n}\right|^p \le e^{\frac{p^2}{2^{k+2}n}}\E\left|S_{2^kn}\right|^p.
        \end{align*}
        Multiplying it for $k = 0, 1, \ldots, m-1$ for some positive integer $m$ gives us
        \begin{align*}
            \E\left|S_{2^mn}\right|^p \le e^{\left(1 - 2^{-m}\right)\frac{p}{2n}}\E|S_n|^p.
        \end{align*}
        By letting $m \to \infty$ we obtain
        \begin{align*}
            \E|G|^p \le e^{\frac{2p}{n}}\E|S_n|^p.
        \end{align*}
        It remains to observe that $e^{-\frac{2p}{n}} \ge 1 - \frac{2p}{n}$ to get the statement with $C(p) = 2p\E|G|^p$.
    \end{proof}

    \subsection{Estimates on constants $C_p$}

    In Theorem \ref{gauss} we get from the proof that \begin{itemize}
        \item $C_p = 2^{\frac{4-p}{2}} \cdot \frac{3p(p-1)(p-2)(p-3)}{128}$ for $p \ge 6$;
        \item $C_p = 2^{\frac{4-p}{2}} \cdot \frac{3p(p-1)(p-2)(p-3)}{128}\widetilde{C}_p$ for $4 < p < 6$;
        \item $C_4 = 2$;
        \item $C_p = \frac{p(p-1)(p-2)(p-3)\left(1 - 2e^{-1}\right)}{6\sqrt{2\pi}}\int_1^{\infty}\frac{t}{(t+1)\sqrt{t+1}}\left(2 + \sqrt{t}\right)^{p-4}e^{-\frac{t+1}{2}}dt$ for $3 < p < 4$;
        \item $C_3 = \frac{9}{32\sqrt{2\pi}}\int_1^{\infty}\frac{t}{(t+1)\sqrt{t+1}}\frac{t-1}{t\sqrt{t}}e^{-\frac{t+1}{2}}dt$.
    \end{itemize}
    Here $\widetilde{C}_p$ is optimal Khintchine constant between the second and $(p-4)$-th moment, $\widetilde{C}_p = \min\left\{2^{\frac{p-6}{2}}, \frac{2^{\frac{p-4}{2}}\Gamma\left(\frac{p-3}{2}\right)}{\sqrt{\pi}}\right\}$ by Haagerup \cite{H}. We ignore the factor $\frac 1 4$ coming from the vectors with $a_1 > \frac{1}{\sqrt{2}}$. In most cases we would like to bound these formulas from below by something much simpler. For $p \ge 6$ we have $C_p \ge \frac{3}{32}(p-3)^42^{-\frac p 2}$, we may also observe that $C_p \sim p^4 2^{-\frac p 2}$ asymptotically with $p \to \infty$, in particular $C_p \xrightarrow{p \to \infty} 0$. For $4 < p < 6$ we have $C_p \ge \frac{3}{128}(p-3)^4\min\left\{2^{-1}, \frac{\Gamma\left(\frac{p-3}{2}\right)}{\sqrt{\pi}}\right\}$, where the last minimum can be easily bounded away from zero by a universal constant. The constant $C_4$ is sharp since $\E|S|^4 = 3 - 2\sum_{i=1}^na_i^4$ and $\E|G|^4 = 3$. The integral in the case $3 < p < 4$ is decreasing in $p$ and for $p = 3$ we get a numerical approximation by $0.037$, which gives $C_p \ge 0.002(p-3)$. Note that $C_p$ vanishes with $p \to 3$. Similarly by numerical approximation of the integral we get $C_3 \ge 10^{-3}$. We could expect that by continuity the optimal values of $C_p$ are bounded away from zero for $p$ close to 3. It could be proved without explicit bound by applying some sort of continuity of $\psi_s''$ in $p$.\\

    Let us return to the asymptotic of $C_p$ with $p \to \infty$. It follows from the proof of Lemma \ref{lemgauss} that the component of order $2^{-\frac p 2}$ comes from the term $\left(1 - a^2\right)^{\frac{p-4}{2}}$, which we bounded from below by $2^{-\frac{p-4}{2}}$ by the assumption $a^2 \le \frac 1 2$. However, our stability results are most interesting for the case where all the coefficients of considered Rademacher sum are small. Then the term $\left(1 - a^2\right)^{\frac{p-4}{2}}$ essentially vanishes and we stay with $C_p \sim p^4 \xrightarrow{p \to \infty} \infty$.\\

    Constants $C_p$ in Theorem \ref{diag} agree with the constants in Theorem \ref{gauss} up to multiplication by a universal constant and base under the exponent $-\frac p 2$ (with the case $n = 2$ as an exception, then our method allows us to provide an explicit constant, but at the cost of adding several lines to the proof in ``baby case", so we decided not to do it). In Theorem \ref{crit} we derived $C_3 = 2^{-6}$. Here it is not clear whether $C_p \xrightarrow{p \to 3} 0$ or not. We cannot use the continuity argument for optimal values of $C_p$ due to asymptotically worse deficit for $p = 3$.

    \subsection{Open questions}

    We stated above that assuming $a_1 \to 0$ we have $C_p \to \infty$ with $p \to \infty$, more specifically $C_p \sim p^4$. But for $p$ being an even integer one can prove that we may assume $C_p = \E|G|^p - 1$. For small $p$ this is a simple computation, but for general even integer $p$ it is more tedious and hence we decided not to present it here. We believe that the same constant would work for arbitrary $p \ge 3$.

    \begin{con}
        For $p \ge 3$ and any Rademacher sum $S = \sum_{i=1}^na_i\e_i$ such that $\sum_{i=1}^na_i^2 = 1$ we have
        \begin{align*}
            \E|S|^p \le \E|G|^p - \left(\E|G|^p - 1\right)\sum_{i=1}^na_i^4.
        \end{align*}
    \end{con}

    In view of the identity \eqref{id2} the deficit obtained in Theorem \ref{diag} agrees with the result of Theorem \ref{gauss}, which is a nice feature suggesting some sort of optimality. The latter theorem holds true also for $p = 3$, which suggests that the deficit obtained in Theorem \ref{crit} should be possible to improve in the spirit of Theorem \ref{diag}. 

    \begin{con}\label{conj}
        For any Rademacher sum $S = \sum_{i=1}^na_i\e_i$ such that $\sum_{i=1}^na_i^2 = 1$ there exists a universal constant $C_3 > 0$ such that
        \begin{align*}
            \E|S|^3 \le \E|S_n|^3 - C_3\sum_{i=1}^n\left(a_i^2 - \frac 1 n\right)^2.
        \end{align*}
    \end{con}

    The natural question is about stability of optimal constants in Khintchine inequality in other cases, uncovered yet. Knowing the differences between the known cases $C_{2, 1}$ and $C_{p, 3}$ it is hard to state a precise conjecture, the question seems to be largely open. In the case $C_{p, 3}$ we would conjecture the ``distributional stability" in the spirit of \cite{ENT1}.

    \begin{con}
        For $p \ge 3$, i.i.d. symmetric random variables $X_1, \ldots, X_n$ close (in some sense) to the Rademacher distribution with $\E|X_1|^2 = 1$ and any $a_1, \ldots, a_n$ such that $\sum_{i=1}^na_i^2$ (perhaps with additional assumptions) the inequality
        \begin{align*}
            \E\left|\sum_{i=1}^na_iX_i\right|^p \le \E|G|^p \quad \text{or} \quad \E\left|\sum_{i=1}^na_iX_i\right|^p \le n^{-\frac p 2}\E\left|\sum_{i=1}^nX_i\right|^p.
        \end{align*}
        holds.
    \end{con}

    \vspace{1em}

    \noindent \textbf{Acknowledgement.} I would like to thank Piotr Nayar for many helpful discussions and Daniel Murawski for suggesting the proof of Proposition \ref{opt}.

    In the first version of this paper Theorem \ref{diag} was proved with worse than present one deficit $C_p\sum_{i=1}^n\left(a_i^2 - \frac 1 n\right)_+\left(\frac 1 n - a_{n+1-i}^2\right)_+$ and Conjecture \ref{conj} was formulated for all $p \ge 3$ instead of $p = 3$ only, i.e. Conjecture \ref{conj} for $p > 3$ became the new version of Theorem \ref{diag}. The only missing step was observing the identity \eqref{id1} an applying its natural consequences, thus I decided to modify the present (still unpublished) paper instead of writing new one. I would like to thank Krzysztof Oleszkiewicz for discussion which helped me to observe this improvement.

    \vspace{1em}

    \noindent Institute of Mathematics \\
    University of Warsaw \\
    02-097, Warsaw, Poland \\
    jj406165@mimuw.edu.pl
	
\end{document}